\documentclass[11pt]{article}
\usepackage{amsmath,amsthm,amsfonts,amssymb,bm,wasysym}
\usepackage{epsfig}
\usepackage[usenames]{color}
\usepackage{verbatim}
\parskip \medskipamount
\newtheorem{theorem}{Theorem}[section]
\newtheorem{definition}[theorem]{Definition}

\numberwithin{equation}{section}
\newtheorem{lemma}[theorem]{Lemma}
\newtheorem{proposition}[theorem]{Proposition}
\newtheorem{corollary}[theorem]{Corollary}

\newtheorem{remark}[theorem]{Remark}
\newtheorem{example}[theorem]{Example}
\newtheorem{claim}{Claim}[section]
\newtheorem{question}[theorem]{Question}
\numberwithin{equation}{section}

\def\N{\mathbb{N}}
\def\Z{\mathbb{Z}}

\def\P{\mathbb{P}}
\def\E{\mathbb{E}}

\renewcommand{\phi}{\varphi}
\renewcommand{\epsilon}{\varepsilon}

\newcommand{\1}{{\text{\Large $\mathfrak 1$}}}

\newcommand{\taver}{t_{\mathrm{ave}}}
\newcommand{\tmix}{t_{\mathrm{mix}}}

\newcommand{\tg}{t_{\mathrm{G}}}
\newcommand{\tsep}{t_{\mathrm{sep}}}
\newcommand{\tl}{t_{\mathrm{L}}}
\newcommand{\thit}{t_{\mathrm{H}}}

\newcommand{\tstop}{t_{\mathrm{stop}}}
\newcommand{\tct}{t_{\mathrm{cts}}}
\newcommand{\tprod}{t_{\mathrm{prod}}}
\newcommand{\tces}{t_{\mathrm{Ces}}}

\newcommand{\til}{\widetilde}

\newcommand{\estart}[2]{\mathbb{E}_{#2}\!\left[#1\right]}

\newcommand\be{\begin{equation}}
\newcommand\ee{\end{equation}}

\begin{document}
\title{\bf Mixing times are hitting times of large sets}

\author{
Yuval Peres\thanks{Microsoft Research, Redmond, Washington, USA; peres@microsoft.com} \and Perla Sousi\thanks{University of Cambridge, Cambridge, UK;   p.sousi@statslab.cam.ac.uk}
}
%\date{}
\maketitle
\begin{abstract}
We consider irreducible reversible discrete time Markov chains on a finite state space.
Mixing times and hitting times are fundamental parameters of the chain. We relate them by showing
that the mixing time of the lazy chain is equivalent to the maximum over initial states $x$ and large sets $A$ of the hitting time of $A$ starting from $x$.
We also prove that the first time when averaging over two consecutive time steps is close to stationarity is equivalent to the mixing time of the lazy version of the chain.
%Besides proving new results, we review the literature on the topic and give simpler proofs to some older results.
\newline
MSC 2010 \emph{subject classifications.}
Primary   60J10.   %Markov chains
\end{abstract}

\section{Introduction}
Mixing times and hitting times are among the most fundamental notions associated with a finite Markov chain. A variety of tools have been developed to estimate both these notions; in particular, hitting times are closely related to potential theory and they can be determined by solving a system of linear equations. In this paper we establish a new connection between mixing times and hitting times for reversible Markov chains (Theorem ~\ref{thm:mix-hit}).

Let $(X_t)_{t \geq 0}$ be an irreducible Markov chain on a finite state space with transition matrix $P$ and stationary distribution $\pi$. For $x,y$ in the state space we write
\[
P^t(x,y)=\P_x(X_t=y),
\]
for the transition probability in $t$ steps.

%A classical result of the theory of Markov chains asserts that the transition probabilities converge to the stationary distribution as $t \to \infty$. A natural question that arises is how long it takes to get ``close'' to stationarity. There are many different ways to measure that.

Let $\displaystyle d(t) = \max_{x}\|P^t(x,\cdot) - \pi \|$, where $\|\mu-\nu\|$ stands for the total variation distance between the two probability measures $\mu$ and $\nu$. Let $\epsilon>0$. The total variation mixing is defined as follows:
\[
\tmix(\epsilon)=\min\{t \geq 0: d(t) \leq \epsilon\}.
\]
We write $P^t_L$ for the transition probability in $t$ steps of the lazy version of the chain, i.e.\ the chain with transition matrix $\frac{P+I}{2}$. If we now let $\displaystyle d_L(t) = \max_{x}\|P^t_L(x,\cdot) - \pi \|$, then we can define the mixing time of the lazy chain as follows:
\begin{align}\label{eq:tmixlazy}
\tl(\epsilon) = \min\{t \geq 0: d_L(t)\leq \epsilon \}.
\end{align}
For notational convenience we will simply write $\tl$ and $\tmix$ when $\epsilon=1/4$.

Before stating our first theorem, we introduce the maximum hitting time of ``big'' sets. Let $\alpha<1/2$, then we define
\[
\thit(\alpha)=\max_{x,A: \pi(A)\geq \alpha} \E_x[\tau_A],
\]
where $\tau_A$ stands for the first hitting time of the set $A$ by the Markov chain with transition matrix~$P$.

It is clear (and we prove it later) that if the Markov chain has not hit a big set, then it cannot have mixed. Thus for every $\alpha>0$, there is a positive constant $c'_\alpha$ so that
\[
\tl \geq c'_\alpha \thit(\alpha).
\]
In the following theorem, we show that the converse is also true when a chain is reversible.
\begin{theorem}\label{thm:mix-hit}
Let $\alpha<1/2$. Then there exist positive constants $c'_\alpha$ and $c_\alpha$ so that for every reversible chain
\[
c'_\alpha \thit(\alpha) \leq \tl \leq c_\alpha\thit(\alpha).
\]
\end{theorem}

\begin{remark}\label{rem:aldous}
\rm{
Aldous in \cite{Aldousmixing} showed that the mixing time, $\tct$, of a continuous time reversible chain is equivalent to $\displaystyle \tprod=\max_{x,A:\pi(A)>0}\pi(A)\E_x[\tau_A]$. The inequality $\tprod\leq c_1 \tct$, for a positive constant $c_1$, which was the hard part in Aldous' proof, follows from Theorem~\ref{thm:mix-hit} and the equivalence $\tl \asymp \tct$ (see \cite[Theorem~20.3]{LevPerWil}). For the other direction we give a new proof in Section~\ref{sec:newproof}.
}
\end{remark}

\begin{remark}\rm{
In Section~\ref{sec:robustness} we present an application of Theorem~\ref{thm:mix-hit} to robustness of the mixing time. Namely, we show that for a finite binary tree, assigning bounded conductances to the edges can only change the mixing time of the lazy random walk on the tree by a bounded factor. However, we note that not all graphs are robust to conductance perturbations. A counterexample is given by Ding and Peres in~\cite{DingPeres}.
}

\end{remark}
To avoid periodicity and near-periodicity issues, one often considers the lazy version of a discrete time Markov chain. In the following theorem we show that averaging over two successive times suffices, i.e.\ $\displaystyle \tl \asymp \taver\left(\tfrac{1}{4}\right)$ where
\[
\taver(\epsilon) = \min\left\{t \geq 0: \max_x \left\|\frac{P^t(x,\cdot)+P^{t+1}(x,\cdot)}{2} - \pi \right\|\leq \epsilon  \right\}.
\]
For notational convenience we will simply write $\taver$ when $\epsilon=1/4$.

%There are cases where a discrete time Markov chain never mixes, namely it takes infinite time to reach stationarity. For instance a periodic Markov chain would never converge to equilibrium in discrete time. To avoid periodicity issues, one usually looks at the lazy version of the chain. This entails though slowing down the chain by a factor of $2$. In applications this is not often desirable. Hence, another notion of mixing is more relevant, the so-called ``average mixing''. We thus define
%\[
%\taver(\epsilon) = \min\left\{t \geq 0: \max_x \left\|\frac{p_t(x,\cdot)+p_{t+1}(x,\cdot)}{2} - \pi \right\|\leq \epsilon  \right\},
%\]
%i.e.\ we average over two consecutive times $t$ and $t+1$.
%For notational convenience we will simply write $\taver$ when $\epsilon=1/4$.
%We will prove the following theorem:
\begin{theorem}\label{thm:ave-lazy}
There exist universal positive constants $c$ and $c'$ so that for every reversible Markov chain
\[
c \tl \leq \taver \leq c' \tl.
\]
\end{theorem}
The problem of relating $\taver$ to the mixing time $\tct$ of the continuous-time chain was raised in Aldous-Fill~\cite{AldFill}, Chapter 4, Open Problem 17. Since   $\tct\asymp \tl$ (see \cite[Theorem~20.3]{LevPerWil}),
Theorem \ref{thm:ave-lazy} gives a partial answer to that problem.

\section{Preliminaries and further equivalences}\label{sec:prel}
In this section we first introduce some more notions of mixing. We will then
state some further equivalences between them mostly in the reversible case and will prove them in later sections. These equivalences will be useful for the proofs of the main results, but are also of independent interest.

The following notion of mixing was first introduced by Aldous in \cite{Aldousmixing} in the continuous time case and later studied in discrete time by Lov{\'a}sz and Winkler in \cite{LovWineff,LWmixing}. It is defined as follows:
\begin{align}\label{eq:deftstop}
\tstop = \max_{x} \min\{\E_x[\Lambda_x]: \Lambda_x \text{ is a stopping time s.t. } \P_x(X_{\Lambda_x} \in \cdot) = \pi(\cdot) \}.
\end{align}
The definition does not make it clear why stopping times achieving the minimum always exist. We will recall the construction of such a stopping time in Section~\ref{sec:stopproof}.

The mixing time of the lazy chain and the average mixing are related to $\tstop$ in the following way.
\begin{lemma}\label{lem:tavertstop}
There exists a uniform positive constant $c_1$ so that for every reversible Markov chain
\[
\taver \leq c_1\tstop.
\]
\end{lemma}

\begin{lemma}\label{lem:lazystop}
There exists a uniform positive constant $c_2$ so that for every reversible Markov chain
\[
\tstop \leq c_2 \tl.
\]
\end{lemma}

We will prove Lemma~\ref{lem:tavertstop} in Section~\ref{sec:stopproof}. Lemma~\ref{lem:lazystop} was proved by Aldous in~\cite{Aldousmixing}, but we include the proof in Section~\ref{sec:proofs} for completeness.

In Section~\ref{sec:proofs} we will show that for any chain we have the following:
\begin{lemma}\label{lem:lazave}
For every $\epsilon \leq 1/4$, there exists a positive constant $c_3$ so that for every Markov chain we have that
\[
\tl(\epsilon) \leq c_3 \taver(\epsilon).
\]
\end{lemma}

\begin{definition}\rm{
We say that two mixing parameters $s$ and $r$ are \textbf{equivalent} for a class of Markov chains $\mathcal{M}$ and write $s \asymp r$, if there exist universal positive constants $c$ and $c'$ so that $cs \leq r \leq c's$ for every chain in $\mathcal{M}$.
We also write $s\lesssim r$ and $s\gtrsim r$ if there exist universal positive constants $c_1$ and $c_2$ such that $s \leq c_1 r$ and $s \geq c_2 r$ respectively.
}
\end{definition}
\begin{proof}[\textbf{Proof of Theorem~\ref{thm:ave-lazy}}]
Lemmas~\ref{lem:tavertstop}, \ref{lem:lazystop} and \ref{lem:lazave} give the desired equivalence between $\taver$ and $\tl$.
\end{proof}

Combining the three lemmas above we get the following:
\begin{corollary}\label{cor:tltstop}
For every reversible Markov chain $\tl$ and $\tstop$ are equivalent.
\end{corollary}

\begin{remark}\rm{
Aldous in \cite{Aldousmixing} was the first to show the equivalence between the mixing time of a continuous time reversible chain and $\tstop$.
}
\end{remark}
We will now define the notion of mixing in a geometric time. The idea of using this notion of mixing to prove Theorem~\ref{thm:mix-hit} was suggested to us by Oded Schramm (private communication June 2008). This notion is also of independent interest, because of its properties that we will prove in this section.

For each $t$, let $Z_t$ be a Geometric random variable taking values in $\{1,2,\ldots\}$ of mean $t$ and success probability $t^{-1}$. We first define
\[
d_G(t) = \max_x\|\P_x(X_{Z_t} = \cdot) - \pi \|.
\]
The geometric mixing is then defined as follows
\[
\tg=\tg(1/4) = \min\{t \geq 0: d_G(t)\leq 1/4 \}.
\]
We start by establishing the monotonicity property of $d_G(t)$.
\begin{lemma}\label{lem:dgmon}
The total variation distance $d_G(t)$ is decreasing as a function of $t$.
\end{lemma}
Before proving this lemma, we note the following standard fact.
\begin{claim}\label{cl:decr}
Let $T$ and $T'$ be two independent positive random variables, also independent of the Markov chain. Then for all $x$
\[
\|\P_x(X_{T+T'}=\cdot) - \pi \| \leq \|\P_x(X_T =\cdot) - \pi \|.
\]
\end{claim}

%\begin{proof}[\textbf{Proof}]
%We have that
%\begin{align*}
%&\|\P_x(X_{T+T'}=\cdot) - \pi \| = \frac{1}{2} \sum_y \left|\P_x(X_{T+T'}=y) - \pi(y)\right|\\
%&= \frac{1}{2} \sum_y
%\left|\sum_z \P_x(X_T=z)\P_z(X_{T'}=y) - \sum_z\pi(z) \P_z(X_{T'}=y)\right| \\
%&\leq \frac{1}{2} \sum_y \sum_z \P_z(X_{T'}=y) \left|\P_x(X_T=z) - \pi(z) \right|= \|\P_x(X_T=\cdot) - \pi \|.
%\end{align*}
%\end{proof}

\begin{proof}[\textbf{Proof of Lemma~\ref{lem:dgmon}}]
We first describe a coupling between the two Geometric random variables, $Z_t$ and $Z_{t+1}$. Let $(U_i)_{i\geq 1}$ be a sequence of i.i.d.\ random variables uniform on $[0,1]$.
We now define
\begin{align*}
Z_t &= \min\left\{i\geq 1: U_i \leq \frac{1}{t}\right\} \text{ and} \\
Z_{t+1} &= \min\left\{i\geq 1: U_i \leq \frac{1}{t+1}\right\}.
\end{align*}
It is easy to see that
\[
Z_{t+1} - Z_t \ \ \text{ is independent of } \ \ Z_t.
\]
Indeed, $\P(Z_{t+1}=Z_t|Z_t) = \frac{t}{t+1}$ and similarly for every $k \geq 1$ we have $\displaystyle \P(Z_{t+1}=Z_t+k|Z_t)= \left(\tfrac{t}{t+1}\right)^{k-1}\left(\tfrac{1}{t+1}\right)^2$.
\newline
We can thus write $Z_{t+1} = (Z_{t+1}-Z_t) + Z_t$, where the two terms are independent.

Claim~\ref{cl:decr} and the independence of $Z_{t+1}-Z_t$ and $Z_t$ give the desired monotonicity of $d_G(t)$.
\end{proof}

\begin{lemma}\label{lem:tgtstop}
For all chains we have that
\[
\tg \leq 4 \tstop + 1.
\]
\end{lemma}

The converse of Lemma~\ref{lem:tgtstop} is true for reversible chains in a more general setting. Namely, let $N_t$ be a random variable independent of the Markov chain and of mean $t$. We define the total variation distance $d_N(t)$ in this setting as follows:
\[
d_N(t)= \max_x \|\P_x(X_{N_t} = \cdot) - \pi\|.
\]
Defining $t_N=t_N(1/4)=\min\{t \geq 0: d_N(t)\leq 1/4 \}$ we have the following:
\begin{lemma}\label{lem:tstopgen}
There exists a positive constant $c_4$ such that for all reversible chains
\[
\tstop \leq c_4 t_N.
\]
In particular, $\tstop \leq c_4 \tg$.
\end{lemma}

We will give the proofs of Lemmas~\ref{lem:tgtstop} and \ref{lem:tstopgen} in Section~\ref{sec:geomix}.

Combining Corollary~\ref{cor:tltstop} with Lemmas~\ref{lem:tgtstop} and \ref{lem:tstopgen} we deduce:
\begin{theorem}\label{thm:tgtlrev}
For a reversible Markov chain $\tg$ and $\tl$ are equivalent.
\end{theorem}

We end this section by stating and proving a result relating $\tmix$ and $\taver$ for any Markov chain. First by the triangle inequality it is clear that always $\taver \leq \tmix$. For the converse we have the following:
\begin{proposition}\label{prop:tavertmix}
Let $0<\delta < 1$. There exists a positive constant $c_5$ so that if $P$ is a transition matrix satisfying $P(x,x)\geq \delta$, for all $x$, then
\[
\tmix \leq c_5 \left(\taver \vee \frac{1}{\delta(1-\delta)} \right).
\]
\end{proposition}

\begin{proof}[\textbf{Proof}]
By the triangle inequality we have that for all $x$
\[
\|P^t(x,\cdot) - \pi\| \leq \left\|\tfrac{1}{2}P^t(x,\cdot)+ \tfrac{1}{2}P^{t+1}(x,\cdot) - \pi \right\| + \left\|\tfrac{1}{2}P^t(x,\cdot) - \tfrac{1}{2}P^{t+1}(x,\cdot) \right\|.
\]
Thus it suffices to show that for all starting points $x$ and all times $t$ there exists a positive constant $c_6$ such that
\begin{align}\label{eq:sufineq}
\|P^t(x,\cdot) - P^{t+1}(x,\cdot)\| \leq \frac{c_6}{\sqrt{t\delta(1-\delta)}},
\end{align}
since $\displaystyle \tmix(\epsilon) \leq c_7\tmix\left(\tfrac{4}{3}\epsilon\right)$, for a positive constant $c_7$ and $\displaystyle \epsilon \leq \tfrac{1}{4}$.

We will now construct a coupling $(X_t,Y_{t+1})$ of $P^t(x,\cdot)$ with $P^{t+1}(x,\cdot)$ such that
\[
\P(X_t \neq Y_{t+1}) \leq \frac{c_6}{\sqrt{t\delta(1-\delta)}}.
\]
Since for all $x$ we have that $P(x,x)\geq \delta$, we can write
\[
P= \delta I + (1-\delta)Q,
\]
for a stochastic matrix $Q$. Let $Z$ be a chain with transition matrix $Q$ that starts from $x$. Let $N_t$ and $N_t'$ be independent and both distributed according to $\mathrm{Bin}(t,1-\delta)$.
We are now going to describe the coupling for the two chains, $X$ and $Y$. Let $(W_s)_{s \geq 1}$ and $(W_s')_{s\geq 1}$ be i.i.d.\ random variables with $\P(W_1=0)=1-\P(W_1=1)=\delta$. 
We define $\displaystyle N_t=\sum_{s=1}^{t}W_s$ and define a process $(N'_t)$ by setting $N_0=0$
and 
\begin{align*}
N'_t =
\begin{cases}
\displaystyle \sum_{s=1}^{t} W_s & \mbox{if } N_{t-1}\neq N_t,   \\ N'_{t+1} &\mbox{if }  N_{t-1} = N'_t.
\end{cases}
\end{align*}
It is straightforward to check that $N'$ is a Markov chain with transition matrix 
\[
A(n,n) = \delta = 1 - A(n,n+1) \ \text{ for all } n\in \N.
\]
 Hence, if for all $t$ we set $X_t = Z_{N'_t}$ and $Y_t=Z_{N_t}$, then it follows that both $X$ and $Y$ are Markov chains with transition matrix $P$.  We now let $\tau = \min\{t\geq 0: X_t=Y_{t+1} \}$. If $W_1=0$, i.e.\
$Y_1=X_0=x$, then $\tau=0$. Otherwise, on the event $W_1=1$, we can bound $\tau$ by
\[
\tau \leq \min\left\{t\geq 0: N'_t=1+ \sum_{s=2}^{t+1}W_s\right\}.
\]
We thus see that $\tau$ is stochastically dominated by the first time that $N'_t-\sum_{s=2}^{t+1}W_s$ hits $1$.
But $N'_t-\sum_{s=2}^{t+1}W_s$ is a symmetric random walk on the real line with transition probabilities $p(k,k+1)=p(k,k-1)=\delta (1-\delta)$ for all $k$. By time $t$ this random walk has moved $L$ number of times, where
\[
L \sim \mathrm{Bin}(t,2\delta(1-\delta)).
\]
By the Chernoff bound for Binomial random variables we get that
\begin{align}\label{eq:chern}
\P\left(L< \frac{t\delta(1-\delta)}{2}\right) \leq e^{-9t\delta(1-\delta)/16}.
\end{align}
Therefore we have that
\[
\P(\tau>t) \leq \P\left(L<\frac{t\delta(1-\delta)}{2}\right) + \P\left(\tau>t, L \geq \frac{t\delta(1-\delta)}{2}\right) \leq   e^{-9t\delta(1-\delta)/16} + \P_0\left(T_1 > \frac{t\delta(1-\delta)}{2}\right),
\]
where $T_1$ denotes the first hitting time of $1$ for a simple random walk on $\Z$. By a classical result for simple random walks on $\Z$ (see for instance \cite[Theorem~2.17]{LevPerWil})
\[
\P_0\left(T_1>\frac{t\delta(1-\delta)}{2}\right) \leq \frac{12\sqrt{2}}{\sqrt{t\delta(1-\delta)}}
\]
and this concludes the proof.
\end{proof}

\begin{remark}\rm{
We note that the upper bound given in Proposition~\ref{prop:tavertmix} is tight, in the sense that both $\taver$ and $\frac{1}{\delta}$ can be attained.
Indeed, for lazy chains $\tmix$ and $\taver$ are equivalent. This follows from the observation above that $\taver \leq \tmix$ and \cite[Proposition~5.6]{LevPerWil}.
For $\delta\leq 1/2$, consider the following transition matrix
%\begin{center}
%\epsfig{file=exampledelta.eps,height=2cm}
%\end{center}
$\left(
  \begin{array}{cc}
    \delta & 1-\delta \\
    1-\delta & \delta \\
  \end{array}
\right).$
It is easy to see that in this case the mixing time is of order $\frac{1}{\delta}$.

}
\end{remark}

\section{Stopping times and a bound for $\taver$}\label{sec:stopproof}
In this section we will first give the construction of a stopping time $T$ that achieves stationarity, i.e.\ for all $x,y$ we have that $\P_x(X_T=y)=\pi(y)$, and also for a fixed $x$ attains the minimum in the definition of $\tstop$ in \eqref{eq:deftstop}, i.e.
\begin{align}\label{eq:meanopt}
\E_x[T] = \min\{\E_x[\Lambda_x]: \Lambda_x \text{ is a stopping time s.t. } \P_x(X_{\Lambda_x} \in \cdot) = \pi(\cdot) \}.
\end{align}
The stopping time that we will construct is called the \textbf{filling rule} and it was first discussed in \cite{BaxChac}. This construction can also be found in \cite[Chapter~9]{AldFill}, but we include it here for completeness.

First for any stopping time $S$ and any starting distribution $\mu$ one can define a sequence of vectors
\begin{align}\label{eq:stopping}
\theta_x(t) = \P_\mu(X_t=x,S\geq t), \ \ \sigma_x(t) = \P_\mu(X_t=x, S=t).
\end{align}
These vectors clearly satisfy
\begin{align}\label{eq:thetasigma}
0 \leq \sigma(t) \leq \theta(t), \ \ (\theta(t) - \sigma(t)) \mathbf{P}= \theta(t+1)\ \forall t; \ \theta(0)=\mu.
\end{align}

We can also do the converse, namely given vectors $(\theta(t),\sigma(t);t\geq0)$ satisfying
\eqref{eq:thetasigma} we can construct a stopping time $S$ satisfying
\eqref{eq:stopping}.
We want to define $S$ so that
\begin{align}\label{eq:defTpr}
\P(S=t| S>t-1, X_t=x,X_{t-1}=x_{t-1},\ldots,X_0=x_0) = \frac{\sigma_x(t)}{\theta_x(t)}.
\end{align}
Formally we define the random variable $S$ as follows: Let $(U_i)_{i \geq 0}$ be a sequence of independent random variables uniform on $[0,1]$. We now define $S$ via
\begin{align*}
S = \inf\left\{t \geq 0: U_t \leq \frac{\sigma_{X_t}(t)}{\theta_{X_t}(t)} \right\}.
\end{align*}
From this definition it is clear that \eqref{eq:defTpr} is satisfied and that
$S$ is a stopping time with respect to an enlarged filtration containing also the random variables $(U_i)_{i \geq 0}$, namely $\mathcal{F}_s\nolinebreak =\nolinebreak \sigma(X_0,U_0,\ldots,X_s,U_s)$.
Also, equations \eqref{eq:stopping} are satisfied. Indeed, setting $x_t=x$ we have
\begin{align*}
\P_{\mu}(X_t = x, S\geq t) = \sum_{x_0,x_1,\ldots,x_{t-1}} \mu(x_0)\prod_{k=0}^{t-1}\left(1-\frac{\sigma_{x_{k}}(k)}{\theta_{x_k}(k)}\right)P(x_k,x_{k+1}) = \theta_x(t),
\end{align*}
since $\theta_y(0) = \mu(y)$ for all $y$ and also $\theta(t+1) = (\theta(t) - \sigma(t))\mathbf{P}$ so cancelations happen. Similarly we get the other equality of \eqref{eq:stopping}.

We are now ready to give the construction of the \textbf{filling rule} $T$. Before defining it formally, we give the \textbf{intuition} behind it. Every state $x$ has a quota which is equal to $\pi(x)$. Starting from an initial distribution $\mu$ we want to calculate inductively the probability that we have stopped so far at each state.
When we reach a new state, we decide to stop there if doing so does not increase the probability of stopping at that state above the quota. Otherwise we stop there with the right probability to exactly fill the quota and we continue with the complementary probability.

We will now give the \textbf{rigorous} construction
by defining the sequence of vectors
$(\theta(t),\sigma(t);t\geq \nolinebreak 0)$ for any starting distribution $\mu$. If we start from $x$, then simply $\mu=\delta_x$. First we set $\theta(0)=\mu$. We now introduce another sequence of vectors $(\Sigma(t);t\geq -1)$. Let $\Sigma_x(-1) = 0$ for all $x$.
We define inductively
\begin{align*}
\sigma_x(t) =
  \begin{cases}
    \theta_x(t), & \mbox{if } \Sigma_x(t-1) + \theta_x(t) \leq \pi(x); \\
    \pi(x) - \Sigma_x(t-1), & \mbox{otherwise.}
  \end{cases}
\end{align*}
Then we let $\Sigma_x(t) = \sum_{s \leq t} \sigma_x(s)$ and define $\theta(t+1)$ via~\eqref{eq:thetasigma}.
Then $\sigma$ will satisfy \eqref{eq:stopping} and $\Sigma_x(t) = \P_\mu(X_T=\nolinebreak x,T\leq \nolinebreak t)$. Also note from the description above it follows that $\Sigma_x(t) \leq \pi(x)$, for all $x$ and all $t$. Thus we get that
\[
\P_\mu(X_T=x) = \lim_{t \to \infty} \Sigma_x(t) \leq \pi(x)
\]
and since both $\P_\mu(X_T=\cdot)$ and $\pi(\cdot)$ are probability distributions, we get that they must be equal. Hence the above construction yielded a stationary stopping time. It only remains to prove the mean-optimality \eqref{eq:meanopt}.
Before doing so we give a definition.
\begin{definition}\label{def:halting}
\rm{
Let $S$ be a stopping time. A state $z$ is called a \textbf{halting} state for the stopping time if $S \leq T_z$ a.s.\, where $T_z$ is the first hitting time of state $z$.
}
\end{definition}
We will now show that the filling rule has a halting state and then the following theorem gives the mean-optimality.

\begin{theorem}[Lov{\'a}sz and Winkler]\label{thm:lovwin}
Let $\mu$ and $\rho$ be two distributions.
Let $S$ be a stopping time such that $\P_\mu(X_S = x)=\rho(x)$ for all $x$. Then $S$ is mean optimal in the sense that
\[
\E_\mu[S] = \min\{\E_\mu[U]: U \text{ is a stopping time s.t. } \P_\mu(X_{U} \in \cdot) = \rho(\cdot) \}
\]
if and only if it has a halting state.
\end{theorem}

Now we will prove that there exists $z$ such that $T \leq T_z$ a.s. For each $x$ we define
\[
t_x = \min\{t: \Sigma_x(t) = \pi(x) \} \leq \infty.
\]
Take $z$ such that $\displaystyle t_z = \max_{x} t_x \leq \infty$. We will show that $T \leq T_z$ a.s.
If there exists a $t$ such that $\P_\mu(T>t , T_z = t)>0$, then $\Sigma_x(t) = \pi(x)$, for all $x$, since the state $z$ is the last one to be filled.
So if the above probability is positive, then we get that
\[
\P_\mu(T \leq t) = \sum_x \Sigma_x(t) =1,
\]
which is a contradiction.
Hence, we obtain that $\P_\mu(T >t, T_z=t) = 0$ and thus by summing over all $t$ we deduce that $\P_\mu(T \leq T_z)=1$.

\begin{proof}[\textbf{Proof of Theorem~\ref{thm:lovwin}}]
We define the exit frequencies for $S$ via $\displaystyle \nu_x = \E_\mu\left[\sum_{k=0}^{S-1}\1(X_k=x)\right]$, for all $x$.
\newline Since $\P_\mu(X_S=\cdot)=\rho(\cdot)$, we can write
\[
\E_\mu\left[\sum_{k=0}^{S}\1(X_k=x)\right] = \E_\mu\left[\sum_{k=0}^{S-1}\1(X_k=x)\right] + \rho(x)=\nu_x+\rho(x).
\]
We also have that
\[
\E_\mu\left[\sum_{k=0}^{S}\1(X_k=x)\right] = \mu(x) + \E_\mu\left[\sum_{k=1}^{S}\1(X_k=x)\right].
\]
Since $S$ is a stopping time, it is easy to see that
\[
\E_\mu\left[\sum_{k=1}^{S}\1(X_k=x)\right] = \sum_y \nu_y P(y,x).
\]
Hence we get that
\begin{align}\label{eq:freq}
\nu_x+\rho(x) = \mu(x) + \sum_y \nu_y P(y,x).
\end{align}
Let $T$ be another stopping time with $\P_\mu(X_T=\cdot)=\rho(\cdot)$ and let $\nu_x'$ be its exit frequencies. Then they would satisfy \eqref{eq:freq}, i.e.
\[
\nu_x'+\rho(x) = \mu(x) + \sum_y \nu_y' P(y,x).
\]
Thus if we set $d=\nu'-\nu$, then $d$ as a vector satisfies
\[
d = d P,
\]
and hence $d$ must be a multiple of the stationary distribution, i.e.\ for a constant $\alpha$ we have that
$d = \alpha \pi$.

Suppose first that $S$ has a halting state, i.e.\ there exists a state $z$ such that $\nu_z=0$. Therefore we get that $\nu_z'= \alpha \pi(z)$, and hence $\alpha \geq 0$. Thus $\nu'_x \geq \nu_x$ for all $x$ and
\[
\E_\mu[T] = \sum_{x} \nu'(x) \geq \sum_{x} \nu_x = \E_\mu[S],
\]
and hence proving mean-optimality.

We will now show the converse, namely that if $S$ is mean-optimal then it should have a halting state. The filling rule was proved to have a halting state and thus is mean-optimal. Hence using the same argument as above we get that $S$ is mean optimal if and only if $\displaystyle \min_x \nu_x = 0$, which is the definition of a halting state.
\end{proof}

Before giving the proof of Lemma~\ref{lem:tavertstop} we state and prove a preliminary result.
\begin{lemma}\label{lem:average}
Let $X$ be a reversible Markov chain on the state space $\Gamma$ and let $L,U$ be positive constants. Let $T$ be a stopping time that achieves stationarity
starting from $x$, i.e.\ $\P_x(X_T = y) = \pi(y)$, for all $y$. For all $y$ and all times $u$ we define $f_y(u)= \frac{1}{2} \P_x(X_u=y, T \leq L) + \frac{1}{2} \P_x(X_{u+1}=y, T \leq L)$. Then there exists $u \leq L + U$ such that
\[
\sum_{y} \frac{f_y(u)^2}{\pi(y)} \leq 1 + \frac{L}{U}.
\]
\end{lemma}

\begin{proof}[\textbf{Proof}]
In this proof we will write $P_{x,y}(t)=P^t(x,y)$ for notational convenience.
We define a measure $\nu$ on $\Gamma \times [0,L]$ by
\[
\nu(\cdot,\cdot) = \P_x(T \leq L, (X_T,T) \in (\cdot,\cdot)).
\]
We define $g_y(u)= \frac{1}{2} \P_x(X_{L+u}=y,T \leq L)+\frac{1}{2} \P_x(X_{L+u+1}=y,T \leq L)$ for $0 \leq u \leq U-1$. By conditioning on $(X_T,T)$ we get
\[
g_y(u) =\frac{1}{2} \sum_{(z,s)}(P_{z,y}(L + u - s)+ P_{z,y}(L + u + 1 - s)) \nu(z,s),
\]
where the sum is over $(z,s)$ in $\Gamma \times [0,L]$.
Thus
\begin{align}\label{eq:gy}
4 \sum_y \pi(y)^{-1} g_y(u)^2 = I_1 + I_2 + I_3 + I_4,
\end{align}
where
\begin{align*}
&I_1= \sum_{\substack{(z_1,s_1)\\(z_2,s_2)}}\sum_y  \pi^{-1}(y) P_{z_1,y}(L+u -s_1) P_{z_2,y}(L+u -s_2) \nu(z_1,s_1) \nu(z_2,s_2),
\\
&I_2= \sum_{\substack{(z_1,s_1)\\(z_2,s_2)}}\sum_y  \pi^{-1}(y) P_{z_1,y}(L+u -s_1) P_{z_2,y}(L+u+1 -s_2) \nu(z_1,s_1) \nu(z_2,s_2),
\\
&I_3= \sum_{\substack{(z_1,s_1)\\(z_2,s_2)}}\sum_y  \pi^{-1}(y) P_{z_1,y}(L+u+1 -s_1) P_{z_2,y}(L+u -s_2) \nu(z_1,s_1) \nu(z_2,s_2) \text{ and}
\\
&I_4= \sum_{\substack{(z_1,s_1)\\(z_2,s_2)}}\sum_y  \pi^{-1}(y) P_{z_1,y}(L+u+1 -s_1) P_{z_2,y}(L+u+1 -s_2) \nu(z_1,s_1) \nu(z_2,s_2).
\end{align*}
By reversibility we have that
\begin{align*}
&I_1=\sum_{\substack{(z_1,s_1)\\(z_2,s_2)}} \pi(z_2)^{-1}P_{z_1,z_2}(2L+2u-s_1-s_2)\nu(z_1,s_1) \nu(z_2,s_2),
\\
&I_2= \sum_{\substack{(z_1,s_1)\\(z_2,s_2)}}\pi(z_2)^{-1}P_{z_1,z_2}(2L+2u +1 -s_1-s_2)\nu(z_1,s_1) \nu(z_2,s_2),
\\
&I_3= \sum_{\substack{(z_1,s_1)\\(z_2,s_2)}}\pi(z_2)^{-1}P_{z_1,z_2}(2L+2u+1-s_1-s_2)\nu(z_1,s_1) \nu(z_2,s_2) \text{ and}
\\
&I_4= \sum_{\substack{(z_1,s_1)\\(z_2,s_2)}}\pi(z_2)^{-1} P_{z_1,z_2}(2L+2u+2-s_1-s_2)\nu(z_1,s_1) \nu(z_2,s_2).
\end{align*}
By considering two cases depending on whether $s_1+s_2$ is odd or even it is elementary to check that
\[
\frac{1}{U}\sum_{u=0}^{U-1} (P_{z_1,z_2}(2L+2u-s_1-s_2) +P_{z_1,z_2}(2L+2u+1-s_1-s_2))     \leq \frac{1}{U}\sum_{u=0}^{2L+2U-1}P_{z_1,z_2}(u),
\]
since $s_1,s_2 \in [0,L]$. Similarly
\[
\frac{1}{U}\sum_{u=0}^{U-1} (P_{z_1,z_2}(2L+2u+2-s_1-s_2) +P_{z_1,z_2}(2L+2u+1-s_1-s_2))     \leq \frac{1}{U}\sum_{u=1}^{2L+2U-1}P_{z_1,z_2}(u).
\]
In this last average we have no dependence on $s_1,s_2$. Hence using \eqref{eq:gy}, the fact that $\nu(z,[0,L])\leq \pi(z)$ for all $z$ and stationarity of $\pi$, we get that
\begin{align*}
\frac{1}{U}\sum_{u=0}^{U-1}\sum_{y} \pi^{-1}(y) g_y(u)^2 & \leq  \frac{1}{4U} \sum_{z_1,z_2} \pi^{-1}(z_2) \left(\sum_{u=0}^{2L+2U-1}P_{z_1,z_2}(u) +\sum_{u=1}^{2L+2U-1}P_{z_1,z_2}(u) \right) \pi(z_1) \pi(z_2) \\
&= 1+L/U.
\end{align*}
This is an upper bound for the average, hence there exists some $u\leq U-1$ such that
\[
\sum_{y}\pi^{-1}(y) g_y(u)^2 \leq 1+ L/U.
\]
\end{proof}

\begin{remark}\rm{
We note that the above lemma uses the same approach as in Aldous \cite[Lemma~38]{Aldousmixing}. Aldous' proof is carried out in continuous time. The proof of Lemma~\ref{lem:average} cannot be done in discrete time for the non lazy version of the chain, since in this case defining $f_y(u)=\P_x(X_u=y, T \leq L)$, we would get that $\displaystyle \sum_y \frac{f_y(u)^2}{\pi(y)} \leq 2 + \frac{L}{U}$. This is where the averaging plays a crucial role.
}
\end{remark}
We now have all the ingredients needed to give the proof of Lemma~\ref{lem:tavertstop}.

\begin{proof}[\textbf{Proof of Lemma~\ref{lem:tavertstop}}]
We fix $x$. Let $T$ be the filling rule as defined at the beginning of this section, which was shown to achieve the minimum appearing in the definition of $\tstop$. Thus, since in the definition of $\tstop$ there is a maximum over the starting points, we have that
\begin{align}\label{eq:Ttstop}
\E_x[T] \leq \tstop.
\end{align}
Let $f_y(u)= \frac{1}{2} \P_x(X_u=y, T \leq L) + \frac{1}{2} \P_x(X_{u+1}=y, T \leq L)$ as appears in Lemma~\ref{lem:average}, where $L$ and $U$ are two positive constants whose precise value will be determined later in the proof and
$u \leq L+U$ is such that
\begin{align}\label{eq:sumfj}
\sum_y \frac{f_y(u)^2}{\pi(y)} \leq 1+\frac{L}{U}.
\end{align}
We then have
\begin{align*}
&\left\|\frac{1}{2}P^u(x,\cdot) + \frac{1}{2} P^{u+1}(x,\cdot) - \pi \right\| = \frac{1}{2} \sum_y \left|\frac{1}{2} P^u(x,y) + \frac{1}{2} P^{u+1}(x,y) - \pi(y)\right|
\\
&\leq \frac{1}{2} \left(\sum_y\left|\frac{1}{2} P^u(x,y) + \frac{1}{2} P^{u+1}(x,y) - f_y(u)\right| + \sum_y |f_y(u) - \pi(y)| \right) \\
&= \frac{1}{2} \left(\P_x(T>L) + \sum_y |f_y(u) - \pi(y)| \right),
\end{align*}
since $f_y(u) \leq \frac{1}{2} P^u(x,y) + \frac{1}{2} P^{u+1}(x,y)$ and $\sum_y f_y(u) = \P_x(T\leq L)$.

By the Cauchy--Schwarz inequality we deduce that
\begin{align*}
&\left(\sum_y|f_y(u) - \pi(y)|\right)^2 = \left(\sum_y \pi(y)^{1/2} \left|\frac{f_y(u)-\pi(y)}{\pi(y)^{1/2}} \right|\right)^2
\leq \sum_y \pi(y)^{-1}(f_y(u)-\pi(y))^2
\\&= \sum_y \pi(y)^{-1} f_y(u)^2 - 2\sum_y f_y(u) + 1 = \sum_y \pi(y)^{-1} f_y(u)^2 - 2\P_x(T\leq L) + 1.
\end{align*}
Using \eqref{eq:sumfj} we get that this last expression is bounded from above by
\[
2\P_x(T>L) + \frac{L}{U}.
\]
Since $\left\|\frac{1}{2}P^{t}(x,\cdot)+\frac{1}{2}P^{t+1}(x,\cdot) - \pi \right\|$ is decreasing in $t$, we conclude that
\[
\left\|\frac{1}{2}P^{L+U}(x,\cdot) + \frac{1}{2} P^{L+U+1}(x,\cdot) - \pi \right\| \leq \frac{1}{2} \left(\P_x(T>L) + \left(2\P_x(T>L) + \frac{L}{U} \right)^{1/2} \right).
\]
If we now take $L=20 \tstop$ and $U=10L$, then by Markov's inequality and \eqref{eq:Ttstop} we get that the total variation distance
\[
\left\|\frac{1}{2}P^{L+U}(x,\cdot) + \frac{1}{2} P^{L+U+1}(x,\cdot) - \pi \right\| \leq \frac{1}{4}.
\]
Thus we get that $\taver \leq L+U = 220 \tstop$ and this concludes the proof of the lemma.
\end{proof}

\section{Proofs of equivalences}\label{sec:proofs}
In Section~\ref{sec:prel} we defined the notion of $\tstop$. In order to prove Lemma~\ref{lem:lazystop} we will first show a preliminary result that compares $\tstop$ to $\tstop^L$, where the latter is defined as
\[
\tstop^L =\max_{x} \min\{\E_x[U_x]: U_x \text{ is a stopping time s.t. } \P_x(X^{L}_{U_x} \in \cdot) = \pi(\cdot) \},
\]
where $X^L$ stands for the lazy version of the chain $X$.

\begin{lemma}\label{lem:equiv}
For every chain we have that
\[
\tstop \leq \frac{1}{2}\tstop^{L}.
\]
\end{lemma}

\begin{proof}[\textbf{Proof}]
Let $X^L$ denote the lazy version of the chain $X$. Then $X^L$ can be realized by viewing $X$ at a $\mathrm{Bin}(t,1/2)$ time, namely let $f(t) \sim \mathrm{Bin}(t,1/2)$, then $X^L_t=X_{f(t)}$ a.s. We can express $f(t)$ as $\displaystyle f(t)=\sum_{j=0}^{t}\xi(j)$, where $(\xi(j))_{j \geq 0}$ are i.i.d.\ fair coin tosses.
Let $T$ be a stopping time for the lazy chain $X^L$.
We enlarge the filtration by adding all the coin tosses. In particular for each $k$ we consider the following filtration:
\[
\mathcal{F}_k = \sigma(X_0,\ldots,X_k, (\xi_j)_{j\geq 0}).
\]
It is obvious that $X$ has the Markov property with respect to the filtration $\mathcal{F}$ too. Also $f(T)$ is a stopping time for that filtration. Indeed,
\begin{align*}
\{f(T)=t\} = \left\{\sum_{j=0}^{T}\xi_j=t \right\} = \bigcup_{\ell \geq t}\left\{T=\ell, \sum_{j=0}^{\ell}\xi_j=t \right\}
\end{align*}
and for each $\ell \geq t$ we have that
\[
\left\{T=\ell, \sum_{j=0}^{\ell}\xi_j=t \right\} \in \sigma(X_0,\ldots,X_t, (\xi_j)_{j \geq 0}),
\]
since on the event $f(\ell)=t$ we have that $X^L_\ell=X_{f(\ell)}=X_t$.
Hence $f(T)$ is a stopping time for $X$ and it achieves stationarity, since for all $x$ and $y$
\[
\P_x\left(X_{f(T)} = y\right) = \P_x(X^L_T = y) = \pi(y),
\]
since $T$ achieves stationarity for the lazy chain.
By Wald's identity for stopping times we get that for all $x$
\[
\E_x[f(T)] = \E_x\left[\sum_{j=1}^{T}\xi(j)\right] = \E_x[T] \E_x[\xi] = \frac{1}{2} \E_x[T].
\]
Hence using a stopping time of the lazy chain $X^L$ achieving stationarity we defined a stopping time for the base chain $X$ achieving stationarity and with expectation equal to half of the original one. Thus for all $x$ we obtain that
\begin{align*}
&\{\E_x[T]: T \text{ stopping time s.t. } \P_x\left(X^L_T=\cdot \right)=\pi\} \\
&\subset \{2 E_x[T']: T' \text{ stopping time s.t. } \P_x\left(X_{T'}=\cdot\right)=\pi \}.
\end{align*}
Therefore taking the minimum concludes the proof.
\end{proof}

Before giving the proof of Lemma~\ref{lem:lazystop} we introduce some notation and a preliminary result that will also be used in the proof of Lemma~\ref{lem:tstopgen}. For any $t$ we let
\begin{align*}
s(t)= \max_{x,y}\left[1-\frac{P^t(x,y)}{\pi(y)} \right] \ \ \text{and} \ \ \bar{d}(t)=\max_{x,y} \| P^t(x,\cdot)-P^t(y,\cdot)\|.
\end{align*}
We will call $s$ the total separation distance from stationarity.

We finally define the separation mixing as follows
\begin{align*}
\tsep = \min\{t \geq 0: s(t) \leq 3/4 \}.
\end{align*}

\begin{lemma}\label{lem:distance}
For a reversible Markov chain we have that
\begin{align*}
d(t) \leq \bar{d}(t) \leq 2d(t) \ \ \text{and} \ \  s(2t)\leq 1-(1-\bar{d}(t))^2.
\end{align*}
\end{lemma}

\begin{proof}[\textbf{Proof}]
A proof of this result can be found in \cite[Chapter~4, Lemma~7]{AldFill} or \cite[Lemma~4.11 and Lemma~19.3]{LevPerWil}.
\end{proof}
%\begin{proof}[\textbf{Proof}]
%The proof of the first set of inequalities can be found in \cite[Lemma~4.11]{LevPerWil}.
%In order to show the second one, we will prove that for a reversible chain
%\begin{align}\label{eq:easy}
%s(2) \leq 1-(1-\bar{d}(1))^2
%\end{align}
%and then for a fixed $t$ by considering the chain $(X_{kt})_{k\geq 0}$, we will get the result.
%We have that for all $x$ and $y$
%\begin{align*}
%&\frac{\P_x(X_{2} = y)}{\pi(y)} = \sum_z \frac{P(x,z)P(z,y)}{\pi(y)} = \sum_z \pi(z)\frac{P(x,z)P(y,z)}{\pi(z)^2}\\
%&\geq \left(\sum_z P(x,z)^{1/2} P(y,z)^{1/2}\right)^2 \geq \left(\sum_z \min(P(x,z),P(y,z))\right)^2 \\
%&=\left(1 - \|P(x,\cdot) - P(y,\cdot) \|  \right)^2 \geq \left(1 - \bar{d}(1) \right)^2,
%\end{align*}
%where for the third equality we used the reversibility and for the first inequality we used the Cauchy-Schwarz inequality.
%\end{proof}

\begin{remark}\label{rem:consequence}\rm{
Lemma~\ref{lem:distance} above gives that $\tsep \leq 2 \tmix$.
}
\end{remark}

\begin{lemma}\label{lem:tstoptsep}
There exists a positive constant $c$ so that for all chains we have
\begin{align*}
\tstop \leq c \tsep
\end{align*}
\end{lemma}

\begin{proof}[\textbf{Proof}]
Fix $t=\tsep$. Then we have that for all $x,y$
\[
P^t(x,y) \geq (1-3/4)\pi(y)=\frac{1}{4} \pi(y).
\]
Hence, we can write
\[
P^t(x,y) = \frac{1}{4}\pi(y) + \frac{3}{4} \nu_x(y),
\]
where for a fixed $x$ we have that $\nu_x$ is a probability measure.
We can now construct a stopping time $S \in \{t,2t,\ldots \}$ so that for all $x$
\[
\P_x(X_S \in \cdot, S=t) = \frac{1}{4}\pi(\cdot)
\]
and by induction on $m$ such that
\[
\P_x(X_S \in \cdot, S=mt) = \left(\frac{3}{4}\right)^{m-1} \frac{1}{4} \pi(\cdot).
\]
Therefore it is clear that $X_S$ is distributed according to $\pi$ and $\E_x[S] = 4t$.
Hence we get that $\tstop \leq 4\tsep$.
\end{proof}

\begin{proof}[\textbf{Proof of Lemma~\ref{lem:lazystop}}]

Let $\tsep^L$ stand for the separation mixing of the lazy chain. Then Lemma~\ref{lem:tstoptsep} gives that
\[
\tstop^L \leq c \tsep^L.
\]
Finally, Lemma~\ref{lem:equiv} and Remark~\ref{rem:consequence} conclude the proof.
\end{proof}

\begin{proof}[\textbf{Proof of Lemma~\ref{lem:lazave}}]
Fix $t$. Let $T$ be a random variable taking values $t$ and $t+1$ each with probability $1/2$, i.e.
\begin{align*}
T =
  \begin{cases}
    t, & \mbox{w.p.} \frac{1}{2} \\
    t+1, & \mbox{w.p.} \frac{1}{2}.
  \end{cases}
\end{align*}
Thus $T$ can be written as $T = Y_1 + t$, where $Y_1$ is Bernoulli with probability $\frac{1}{2}$.
Then we have that for all $x$ and $y$
\[
\P_x(X_T = y) = \frac{1}{2} \P_x(X_t = y) + \frac{1}{2} \P_x(X_{t+1}=y).
\]
Let $\displaystyle Z \sim \mathrm{Bin}(3t,\tfrac{1}{2})$. Then we can write $Z$ as $Z = Y_1 + Z_1$, where $Z_1$ is distributed according to $\displaystyle \mathrm{Bin}(3t-1,\tfrac{1}{2})$ and is independent of $Y_1$. Therefore $Z$ can be expressed as the sum of two independent random variables, $Z = T + (Z_1 - t)$. (With high probability $Z_1 - t \approx (Z_1 - t)_+$.)
We fix $x$. By the triangle inequality for the total variation distance, we obtain
\begin{align*}
\|\P_x(X_Z = \cdot) - \pi \| \leq \|\P_x(X_{T+(Z_1 - t)_+} = \cdot) - \pi \| + \|\P_x(X_{T + (Z_1- t)} = \cdot) - \P_x(X_{T + (Z_1 - t)_+} = \cdot) \|.
\end{align*}
Since $T$ and $(Z_1-t)_+$ are independent and $(Z_1 - t)_+ \geq 0$, by the monotonicity of the total variation distance Claim~\ref{cl:decr}, we deduce that
\begin{align}\label{eq:mon}
\|\P_x(X_{T+(Z_1 - t)_+} = \cdot) - \pi \| \leq \|\P_x(X_{T} = \cdot) - \pi \|.
\end{align}
It is easy to see that
\begin{align}\label{eq:expon}
\|\P_x(X_{T + (Z_1- t)} = \cdot) - \P_x(X_{T + (Z_1 - t)_+} = \cdot) \| \leq \P_x(Z_1<t)\leq e^{-ct},
\end{align}
for a positive constant $c$, since $Z_1$ follows the Binomial distribution.
Hence by \eqref{eq:mon} and \eqref{eq:expon} we get that
\begin{align}\label{eq:avelazy}
\|\P_x(X_Z = \cdot) - \pi \| \leq \|\P_x(X_{T} = \cdot) - \pi \| + e^{-ct}
\end{align}
The mixing time for the lazy chain was defined in \eqref{eq:tmixlazy}. Equivalently it is given by
\[
\tl(\epsilon) = \min\left\{t: \max_i \|\P_i(X_{Z'}=\cdot) - \pi\| <\epsilon \right\},
\]
where $Z'$ is distributed according to $\mathrm{Bin}(t,1/2)$. Thus
\[
\tl(\epsilon) \leq 3\min\left\{t: \max_i \|\P_i(X_{Z}=\cdot) - \pi\| <\epsilon \right\}.
\]
Finally, from \eqref{eq:avelazy} we get that there exists a constant $c_2>0$ such that
\[
\tl\left(\tfrac{4}{3}\epsilon\right)\leq c_2 \taver(\epsilon).
\]
But $\displaystyle \tl\left(\tfrac{4}{3}\epsilon\right) \geq c_3 \tl(\epsilon)$, since $\displaystyle \epsilon \leq \tfrac{1}{4}$ and this concludes the proof.
%Indeed it is clear that $\displaystyle \tl\left(\tfrac{4}{3}\epsilon\right)\leq \tl(\epsilon)$. For the other direction we have that
%\[
%d_L\left(2\tl\left(\tfrac{4}{3}\epsilon\right)\right) \leq 2d_L\left(\tl\left(\tfrac{4}{3}\epsilon\right)\right)^2 \leq 2 \left(\tfrac{4}{3}\epsilon\right)^2 \leq \tfrac{8}{9} \epsilon <\epsilon,
%\]
%since $\displaystyle \epsilon\leq \tfrac{1}{4}$.
\end{proof}

\section{Mixing at a geometric time}\label{sec:geomix}
Before giving the proof of Lemma~\ref{lem:tgtstop} we state two easy facts about total variation distance.

\begin{claim}\label{cl:total}
Let $Y$ be a discrete random variable with values in $\N$ and satisfying
\[
\P(Y=j) \leq c, \text{ for all } j>0 \ \ \text{and} \ \ \P(Y=j) \text{ is decreasing in } j,
\]
where $c$ is a positive constant.
Let $Z$ be an independent random variable with values in $\N$. Then
\begin{align}\label{eq:claim}
\|\P(Y+Z = \cdot) - \P(Y=\cdot) \| \leq c \E[Z].
\end{align}
\end{claim}

\begin{proof}[\textbf{Proof}]
Using the definition of total variation distance and the assumption on $Y$ we have for all $k \in \N$
\begin{align*}
\|\P(Y+k = \cdot) - \P(Y=\cdot)  \| = \sum_{j:  \P(Y = j) \geq \P(Y+k=j)}( \P(Y=j) - \P(Y+k = j)) \leq k c.
\end{align*}
Finally, since $Z$ is independent of $Y$, we obtain~\eqref{eq:claim}.
\end{proof}

The coupling definition of total variation distance gives the following:
\begin{claim}\label{cl:function}
Let $X$ be a Markov chain and $W$ and $V$ be two random variables with values in $\N$. Then
\begin{align*}
\|\P(X_W = \cdot) - \P(X_V=\cdot) \| \leq \|\P(W = \cdot) - \P(V = \cdot) \|.
\end{align*}
\end{claim}

\begin{proof}[\textbf{Proof of Lemma~\ref{lem:tgtstop}}]
We fix $x$. Let $\tau$ be a stationary time, i.e.\ $\P_x(X_{\tau}=\nolinebreak \cdot)\nolinebreak =\nolinebreak \pi$. Then $\tau+s$ is also a stationary time for all $s\geq 1$. Hence, if $Z_t$ is a Geometric random variable independent of $\tau$, then $Z_t + \tau$ is also a stationary time, i.e.\ $\P_x(X_{Z_t + \tau} = \cdot) = \pi$. Since $Z_t$ and $\tau$ are independent, and $Z_t$ satisfies the assumptions of Claim~\ref{cl:total}, we get
\begin{align}\label{eq:vargeo}
\|\P_x(Z_t + \tau = \cdot) - \P_x(Z_t = \cdot) \| \leq \frac{\E_x[\tau]}{t}.
\end{align}
From Claim~\ref{cl:function}, we obtain
\begin{align*}
\|\P_x(X_{Z_t + \tau} = \cdot) - \P_x(X_{Z_t} = \cdot) \| \leq \|\P_x(Z_t + \tau = \cdot) - \P_x(Z_t = \cdot) \| \leq \frac{\E_x[\tau]}{t},
\end{align*}
and since $\P_x(X_{Z_t + \tau} = \cdot) = \pi$, taking $t \geq 4 \E_x[\tau]$ concludes the proof.
\end{proof}

%second proof
%We fix $x$. Let $\tau$ be a stationary time, i.e.\ $\P_x(X_{\tau}=\nolinebreak \cdot)\nolinebreak =\nolinebreak \pi$. Then $\tau+s$ is also a stationary time for all $s\geq 1$. Hence we have that
%\begin{align*}
%t \pi(y) = \sum_{s=1}^{\infty} \left(1-\frac{1}{t} \right)^{s-1} \P_x(X_{\tau+s}=y) = \sum_{s=1}^{\infty} \left(1-\frac{1}{t} \right)^{s-1}
%\sum_{\ell=0}^{\infty} \P_x(X_{\ell+s}=y,\tau=\ell)\\
%= \sum_{m=1}^{\infty} \sum_{\ell=0}^{m-1}  \left(1-\frac{1}{t} \right)^{m-\ell-1}\P_x(X_{m}=y,\tau=\ell) \geq \sum_{m=1}^{\infty} \left(1-\frac{1}{t} \right)^{m-1}\P_x(X_{m}=y,\tau<m).
%\end{align*}
%As in Section~\ref{sec:prel}, let $Z_t$ be a geometric random variable of success probability $\displaystyle \tfrac{1}{t}$. Then
%\[
%\nu_t(y):= \P_x(X_{Z_t}=y)=\sum_{s=1}^{\infty}\frac{1}{t}\left(1-\frac{1}{t} \right)^{s-1} \P_x(X_s=y).
%\]
%Thus we obtain
%\begin{align*}
%t\nu_t(y)-t\pi(y) &\leq \sum_{s=1}^{\infty}\left(1-\frac{1}{t} \right)^{s-1}\left(\P_x(X_s=y) - \P_x(X_{s}=y,\tau<s)\right)\\
%& =  \sum_{s=1}^{\infty}\left(1-\frac{1}{t} \right)^{s-1} \P_x(X_{s}=y,\tau \geq s).
%\end{align*}
%Summing over all $y$ such that $t\nu_t(y)-t\pi(y)>0$, we get that
%\begin{align*}
%t \|\nu_t-\pi\| \leq \sum_{s=1}^{\infty}\left(1-\frac{1}{t} \right)^{s-1} \P_x(\tau \geq s) \leq \E_x[\tau].
%\end{align*}
%Therefore, if we take $t \geq 4 \E_x[\tau]$, then we get that
%\[
%\|\nu_t-\pi\| \leq \frac{1}{4},
%\]
%and hence $\tg \leq 4 \tstop+1$.

Recall from Section~\ref{sec:prel} the definition of $N_t$ as a random variable independent of the Markov chain and of mean $t$. We also defined
\[
d_N(t) = \max_x \|\P_x(X_{N_t} = \cdot) - \pi\|.
\]
Let $N_t^{(1)},N_t^{(2)}$ be i.i.d.\ random variables distributed as $N_t$ and set $V_t= N_t^{(1)}+N_t^{(2)}$. We now define
\[
s_N(t) = \max_{x,y}\left[1 - \frac{\P_x(X_{V_t}=y)}{\pi(y)} \right] \ \ \text{and} \ \ \bar{d}_N(t) = \max_{x,y} \|\P_x(X_{N_t} = \cdot) - \P_y(X_{N_t} = \cdot)\|.
\]
When $N$ is a geometric random variable we will write $d_G(t)$ and $\bar{d}_G(t)$ respectively.

\begin{lemma}\label{lem:septotal}
For all $t$ we have that
\[
d_N(t) \leq \bar{d}_N(t) \leq 2 d_N(t) \ \ \text{and} \ \ s_N(t) \leq 1 - (1 - \bar{d}_N(t))^2.
\]
\end{lemma}

\begin{proof}[\textbf{Proof}]
Fix $t$ and consider the chain $Y$ with transition matrix $Q(x,y)=\P_x(X_{N_t}=y)$. Then $Q^2(x,y)=\P_x(X_{V_t}=y)$, where $V_t$ is as defined above. Thus, if we let
\[
s_Y(u)=\max_{x,y}\left[1 - \frac{Q^u(x,y)}{\pi(y)} \right] \ \ \text{and} \ \ \bar{d}_Y(u)= \max_{x,y}\|\P_x(Y_u=\cdot) - \P_y(Y_u=\cdot)\|,
\]
then we get that $s_N(t)=s_Y(2)$ and $\bar{d}_N(t)=\bar{d}_Y(1)$. Hence, the lemma follows from Lemma~\ref{lem:distance}.
\end{proof}

We now define
\[
t_{s,N} = \min\left\{t\geq 0: s_N(t) \leq \tfrac{3}{4} \right\}.
\]

\begin{lemma}\label{lem:tsep}
There exists a positive constant $c$ so that for every chain
\[
\tstop \leq c t_{s,N}.
\]
\end{lemma}

\begin{proof}[\textbf{Proof}]
%In this proof we use a similar argument as in the proof of Lemma~\ref{lem:lazystop}.

Fix $t=t_{s,N}$. Consider the chain $Y$ with transition kernel $Q(x,y)=\P_x(X_{V_t}=y)$, where $V_t$ is as defined above.
\newline
By the definition of $s_N(t)$ we have that for all $x$ and $y$
\[
Q(x,y) \geq (1-s_N(t)) \pi(y) \geq \frac{1}{4} \pi(y).
\]
Thus, in the same way as in the proof of Lemma~\ref{lem:tstoptsep}, we can construct a stopping time
$S$ such that $Y_S$ is distributed according to $\pi$ and $\E_x[S]=4$ for all $x$.
%Hence, we can write
%\[
%Q(x,y) = \frac{1}{4} \pi(y) + \frac{3}{4} \nu_x(y),
%\]
%where $\nu_x(\cdot)$ is a probability measure.
%\newline
%We can thus construct a stopping time $S \in \{1,2,\ldots\}$ such that for all $x$
%\[
%\P_x(Y_S \in \cdot, S=1) = (1-3/4)\pi(\cdot)
%\]
%and by induction on $k$ such that
%\[
%\P_x(Y_S \in \cdot, S=k) = (3/4)^{k-1} (1/4)\pi(\cdot).
%\]
%Hence, it is clear that $Y_S$ is distributed according to $\pi$ and $\E_x[S]=4$ for all $x$.

Let $V_t^{(1)},V_t^{(2)},\ldots$ be i.i.d.\ random variables distributed as $V_t$. Then we can write $Y_u = X_{V_t^{(1)}+\ldots +V_t^{(u)}}$. If we let $T=V_t^{(1)}+\ldots+V_t^{(S)}$, then $T$ is a stopping time for $X$ such that $\mathcal{L}(X_T)=\pi$ and by Wald's identity for stopping times we get that for all $x$
\[
\E_x[T]=\E_x[S] \E[V_t] =8t.
\]
Therefore we proved that
\[
\tstop \leq 8t_{s,N}.
\]
\end{proof}

%Consider the chain $Y$ with transition kernel $p(x,y)=\P_x(X_{V_t}=y)$, where $V_t = N_t^{(1)}+N_t^{(2)}$ and $N_t^{(1)}$ and $N_t^{(2)}$ are independent. Let $V_t^{(1)},V_t^{(2)},\ldots$ be i.i.d.\ random variables distributed as $V_t$. Then we can write $Y_u = X_{V_t^{(1)}+\ldots +V_t^{(u)}}$. For $Y$ we can construct a stopping time $S$ such that $\E[S]=\frac{1}{1-s_N(t)}$ and $\mathcal{L}(Y(S))=\pi$. This can be done as follows. We have that for all $j,k$
%\[
%\P_j(X(V_t)=k)\geq (1-3/4)\pi(k).
%\]
%{\color{red} Fix the construction. The expectation of $S$ is not correct.}
%We can construct a stopping time $S \in \{t,2t,\ldots\}$ such that
%\[
%\P(Y_S \in \cdot, S=t) = (1-3/4)\pi(\cdot)
%\]
%and by induction on $m$ such that
%\[
%\P(Y_S \in \cdot, S=mt) = (3/4)^{m-1} (1/4)\pi(\cdot).
%\]
%Hence, it is clear that $Y_S$ is distributed according to $\pi$. This in turn gives a stopping time $T$ for $X$, namely $V_t^{(1)}+\ldots+V_t^{(S)}$ such that $\mathcal{L}(X_T)=\pi$ and by Wald's identity for stopping times we get that
%\[
%\E[T]=\E[S] \E[V_t] = \frac{2t}{1-s_N(t,2)} \leq 8 t.
%\]
%Thus we get that
%\[
%\tstop \leq c\tsep.
%\]

\begin{proof}[\textbf{Proof of Lemma~\ref{lem:tstopgen}}]
From Lemma~\ref{lem:septotal} we get that
\[
t_{s,N} \leq 2 t_N.
\]
Finally Lemma~\ref{lem:tsep} completes the proof.
\end{proof}

\begin{remark}\rm{
Let $N_t$ be a uniform random variable in $\{1,\ldots,t\}$ independent of the Markov chain.
The mixing time associated to $N_t$ is called Ces{\`a}ro mixing and it has been analyzed by Lov{\'a}sz and Winkler in \cite{LWmixing}. From \cite[Theorem~6.15]{LevPerWil} and the lemmas above we get the equivalence between the Ces{\`a}ro mixing and the mixing of the lazy chain in the reversible case. In Section~\ref{sec:cesaro} we show that the Ces{\`a}ro mixing time is equivalent to $\tg$ for all chains.
}
\end{remark}

\begin{remark}\rm{
From the remark above we see that the mixing at a geometric time and the Ces{\`a}ro mixing are equivalent for a reversible chain. The mixing
at a geometric time though has the advantage that its total variation distance, namely $d_G(t)$, has the monotonicity property Lemma~\ref{lem:dgmon}, which is not true for the corresponding total variation distance for the Ces{\`a}ro mixing.
}
\end{remark}

Recall that $\displaystyle \bar{d}(t)=\max_{x,y}\|\P_x(X_t=\cdot) - \P_y(X_t=\cdot)\|$ is submultiplicative as a function of $t$ (see for instance \cite[Lemma~4.12]{LevPerWil}).
In the following lemma and corollary, which will be used in the proof of Theorem~\ref{thm:mix-hit}, we show that $\bar{d}_G$ satisfies some sort of submultiplicativity.
\begin{lemma}\label{lem:submult}
Let $\beta<1$ and let $t$ be such that $\bar{d}_G(t) \leq \beta$.
Then for all $k \in \N$ we have that
\[
\bar{d}_G(2^kt) \leq \left(\frac{1+\beta}{2}\right)^k \bar{d}_G(t).
\]
\end{lemma}

\begin{proof}[\textbf{Proof}]
As in the proof of Lemma~\ref{lem:dgmon} we can write
$Z_{2t} = (Z_{2t}-Z_t) + Z_t$, where $Z_{2t}-Z_t$ and $Z_t$ are independent. Hence it is easy to show (similar to the case for deterministic times) that
\begin{align}\label{eq:dbar}
\bar{d}_G(2t) \leq \bar{d}_G(t) \max_{x,y}\|\P_x(X_{Z_{2t}-Z_t} = \cdot) - \P_y(X_{Z_{2t}-Z_t} = \cdot)\|.
\end{align}
By the coupling of $Z_{2t}$ and $Z_t$ it is easy to see that $Z_{2t}-Z_t$ can be expressed as follows:
\[
Z_{2t}-Z_t = (1-\xi) + \xi G_{2t},
\]
where $\xi$ is a Bernoulli$\displaystyle (\tfrac{1}{2})$ random variable and $G_{2t}$ is a Geometric random variable of mean $2t$ independent of $\xi$. By the triangle inequality we get that
\[
\|\P_x(X_{Z_{2t}-Z_t} = \cdot) - \P_y(X_{Z_{2t}-Z_t} = \cdot)\| \leq \frac{1}{2} + \frac{1}{2} \|\P_x(X_{G_{2t}} = \cdot) - \P_y(X_{G_{2t}} = \cdot)\| = \frac{1}{2}+ \frac{1}{2}\bar{d}_G(2t),
\]
and hence \eqref{eq:dbar} becomes
\[
\bar{d}_G(2t) \leq \bar{d}_G(t)\left(\frac{1}{2}+\frac{1}{2}\bar{d}_G(2t) \right) \leq \frac{1}{2}\bar{d}_G(t)\left(1+\bar{d}_G(t) \right),
\]
where for the second inequality we used the monotonicity property of $\bar{d}_G$ (same proof as for $d_G(t)$).
Thus, since $t$ satisfies $\bar{d}_G(t)\leq \beta$, we get that
\[
\bar{d}_G(2t) \leq \left(\frac{1+\beta}{2} \right) \bar{d}_G(t),
\]
and hence iterating we deduce the desired inequality.
\end{proof}
Combining Lemma~\ref{lem:submult} with Lemma~\ref{lem:septotal} we get the following:
\begin{corollary}\label{cor:submult}
Let $\beta<1$. 
If \ $t$ is such that $d_G(t)\leq \beta/2$, then for all $k$ we have that
\[
d_G(2^kt) \leq 2 \left(\frac{1+\beta}{2}\right)^k d_G(t).
\]
Also if $d_G(t)\leq \alpha <1/2$, then there exists a constant $c=c(\alpha)$ depending only on $\alpha$, such that $d_G(ct)\leq 1/4$.
\end{corollary}

\section{Hitting large sets}\label{sec:hitting}
In this section we are going to give the proof of Theorem~\ref{thm:mix-hit}.
We first prove an equivalence that does not require reversibility.

\begin{theorem}\label{thm:tgthit}
Let $\alpha<1/2$. For every chain $\tg \asymp \thit(\alpha)$. (The implied constants depend on $\alpha$.)
\end{theorem}

\begin{proof}[\textbf{Proof}]
We will first show that $\tg \geq c \thit(\alpha)$.
By Corollary~\ref{cor:submult}  there exists $k=k(\alpha)$ so that $d_G(2^{k}\tg)\leq \frac{\alpha}{2}$. Let $t=2^k\tg$.
Then for any starting point $x$ we have that
\[
\P_x(X_{Z_t} \in A) \geq \pi(A) - \alpha/2 \geq \alpha/2.
\]
Thus by performing independent experiments, we deduce that $\tau_A$ is stochastically dominated by $\sum_{i=1}^{N}G_i$, where $N$ is a Geometric random variable of success probability $\alpha/2$ and the $G_i$'s are independent Geometric random variables of success probability $\frac{1}{t}$. Therefore for any starting point $x$ we get that
\[
\E_x[\tau_A] \leq \frac{2}{\alpha}t,
\]
and hence this gives that
\[
\max_{x,A:\pi(A)\geq \alpha} \E_x[\tau_A] \leq  \frac{2}{\alpha} 2^k \tg.
\]
In order to show the other direction, let $t' < \tg$.
Then $d_G(t') > 1/4$. For a given $\alpha < 1/2$, we fix $\gamma \in (\alpha,1/2)$. From Corollary~\ref{cor:submult} we have that there exists a positive constant $c=c(\gamma)$ such that
\[
d_G(ct')>\gamma.
\]
Set $t=ct'$.
Then there exists a set $A$ and a starting point $x$ such that
\[
\pi(A) - \P_x(X_{Z_t} \in A) > \gamma,
\]
and hence $\pi(A)>\gamma$,
or equivalently
\[
\P_x(X_{Z_t} \in A) < \pi(A) -\gamma.
\]
We now define a set $B$ as follows:
\[
B= \{y: \P_y(X_{Z_t} \in A)\geq \pi(A)-\alpha\},
\]
where $c$ is a constant smaller than $\alpha$.
Since $\pi$ is a stationary distribution, we have that
\begin{align*}
\pi(A) = \sum_{y \in B}\P_y(X_{Z_t} \in A)\pi(y) + \sum_{y \notin B}\P_y(X_{Z_t} \in A)\pi(y) \leq \pi(B) + \pi(A) - \alpha,
\end{align*}
and hence rearranging, we get that
\[
\pi(B) \geq \alpha.
\]
We will now show that for a constant $\theta$ to be determined later we have that
\begin{align}\label{eq:goal}
\max_z \E_z[\tau_B] >\theta t.
\end{align}
We will show that for a $\theta$ to be specified later, assuming
\begin{align}\label{eq:contr}
\max_z \E_z[\tau_B] \leq \theta t
\end{align}
will yield a contradiction.
\newline
By Markov's inequality, \eqref{eq:contr} implies that
\begin{align}\label{eq:exper}
\P_x(\tau_B \geq 2 \theta t) \leq \frac{1}{2}.
\end{align}
For any positive integer $M$ we have that
\[
\P_x(\tau_B \geq 2M\theta t) = \P_x(\tau_B\geq 2M\theta t|\tau_B \geq 2(M-1)\theta t) \P_x(\tau_B\geq 2(M-1)\theta t),
\]
and hence iterating we get that
\begin{align}\label{eq:ineqtau}
\P_x(\tau_B \geq 2M\theta t) \leq \frac{1}{2^M}.
\end{align}
By the memoryless property of the Geometric distribution and the strong Markov property applied at the stopping time $\tau_B$, we get that
\begin{align*}
&\P_x(X_{Z_t} \in A) \geq \P_x(\tau_B \leq 2 \theta M t, Z_t \geq \tau_B,X_{Z_t} \in A) \\
&\geq \P_x(\tau_B \leq 2\theta M t, Z_t \geq \tau_B)\P_x(X_{Z_t} \in A| \tau_B \leq 2\theta M t, Z_t \geq \tau_B)
\\
&\geq \P_x(\tau_B \leq 2\theta M t) \P_x(Z_t \geq \lfloor 2\theta M t\rfloor) \left( \inf_{w \in B} \P_w(X_{Z_t} \in A)\right),
\end{align*}
where in the last inequality we used the independence between $Z$ and $\tau_B$.
But since $Z_t$ is a Geometric random variable, we obtain that
\[
\P_x(Z_t \geq \lfloor 2\theta M t\rfloor) \geq  \left(1-\frac{1}{t}\right)^{2\theta M t},
\]
which for $2\theta M t>1$ gives that
\begin{align}\label{eq:geom}
\P_x(Z_t \geq \lfloor 2\theta M t\rfloor) \geq 1-2\theta M.
\end{align}
(\eqref{eq:contr} implies that $\theta t \geq 1$, so certainly $2\theta M t >1$.)
\newline
We now set $\theta=\frac{1}{2M2^{M}}$. Using \eqref{eq:exper} and \eqref{eq:geom} we deduce that
\begin{align*}
\P_x(X_{Z_t} \in A) \geq \left(1-2^{-M} \right)^2 (\pi(A) - \alpha).
\end{align*}
Since $\gamma > \alpha$, we can
take $M$ large enough so that $\left(1-2^{-M}\right)^2 (\pi(A)-\alpha) > \pi(A)-\gamma$, and we get a contradiction to \eqref{eq:contr}. \

Thus \eqref{eq:goal} holds; since $\pi(B) \geq \alpha$, this completes the proof.
\end{proof}

\begin{proof}[\textbf{Proof of Theorem~\ref{thm:mix-hit}}]
Combining Theorem~\ref{thm:tgtlrev} with Theorem~\ref{thm:tgthit} gives the result in the reversible case.
\end{proof}
%From the equivalence proved by Aldous in \cite{Aldousmixing}
%\[
%\tl \asymp \thit,
%\]
%we get that
%\[
%\tl \geq b \max_{x,A: \, \pi(A)\geq b} \E_x[\tau_A].
%\]
%
%A direct way to see that is the following:
%Let $t=\tl(\alpha/2) \leq \left(\log_2\frac{2}{\alpha}\right) \tl$ and let $A$ be a set with $\pi(A)\geq \alpha$.
%Then for any starting point $x$ we have that
%\[
%P^t(x,A) \geq \pi(A) - \alpha/2 \geq \alpha/2.
%\]
%Thus by performing independent experiments, we deduce that $\tau_A$ is stochastically dominated by $t \mathrm{Geo}(\alpha/2)$, where $\mathrm{Geo}(\alpha/2)$ stands for a Geometric random variable of success probability $\alpha/2$. Therefore for any starting point $x$ we get that
%\[
%\E_x[\tau_A] \leq \frac{2}{\alpha}t,
%\]
%and hence this gives that
%\[
%\max_{x,A:\pi(A)\geq \alpha} \E_x[\tau_A] \leq \frac{2}{\alpha} \left(\log_2 \frac{2}{\alpha} \right) \tl.
%\]

\section{Equivalence between Ces{\`a}ro mixing and $\tg$}\label{sec:cesaro}
In this section we will show that the notion of mixing at a geometric time defined in Section~\ref{sec:prel} and the Ces{\`a}ro mixing used by Lov{\'a}sz and Winkler~\cite{LWmixing} are equivalent for all chains. First, let us recall the definition of Ces{\`a}ro mixing. Let $U_t$ be a random variable independent of the chain uniform on $\{1,\ldots,t\}$. We define
\begin{align*}
\tces = \min\left\{t\geq 0: \max_x \|\P_x(X_{U_t} = \cdot) - \pi \| \leq \frac{1}{4} \right\}.
\end{align*}

\begin{proposition}\label{prop:tcestg}
For all chains $\tg \asymp \tces$.
\end{proposition}

\begin{proof}[\textbf{Proof}]
For each $s$, let $U_s$ be a uniform random variable in $\{1,\ldots,s\}$ and $Z_s$ an independent geometric random variable of mean $s$.

We will first show that there exists a positive constant $c_1$ such that
\begin{align}\label{eq:upperbound}
\tces \leq c_1 \tg.
\end{align}
Let $t = \tg(1/8)$, then for all $x$
\begin{align}\label{eq:deftotal}
\|\P_x(X_{Z_t} = \cdot) - \pi \| \leq \frac{1}{8}.
\end{align}
From Claims~\ref{cl:total} and~\ref{cl:function} we get that
\[
\| \P_x(X_{U_{8t}}=\cdot) - \P_x(X_{U_{8t} + Z_t} = \cdot) \| \leq \|\P_x(U_{8t}=\cdot) - \P_x(U_{8t} + Z_t = \cdot) \| \leq \frac{1}{8}.
\]
By the triangle inequality for total variation we deduce
\begin{align*}
\|\P_x(X_{U_{8t}}=\cdot) - \pi \| \leq \| \P_x(X_{U_{8t}}=\cdot) - \P_x(X_{U_{8t} + Z_t} = \cdot) \| + \|\P_x(X_{U_{8t} + Z_t} = \cdot) - \pi\|
\end{align*}
From~\eqref{eq:deftotal} and Claim~\ref{cl:decr} it follows that
\[
\|\P_x(X_{U_{8t} + Z_t} = \cdot) - \pi\| \leq \|\P_x(X_{Z_t} = \cdot) - \pi\| \leq \frac{1}{8}.
\]
Hence, we conclude
\[
\|\P_x(X_{U_{8t}}=\cdot) - \pi \| \leq \frac{1}{4},
\]
which gives that $\tces \leq 8 t$. From Corollary~\ref{cor:submult} we get that there exists a constant $c$ such that
$\tg(1/8) \leq c\tg$ and this concludes the proof of~\eqref{eq:upperbound}.

We will now show that there exists a positive  constant $c_2$ such that
\[
\tg \leq c_2 \tces.
\]
Let $t= \tces$, i.e.\ for all $x$
\begin{align}\label{eq:cesaro}
\|\P_x(X_{U_t} = \cdot) - \pi \| \leq \frac{1}{4}.
\end{align}
From Claims~\ref{cl:total} and~\ref{cl:function} we get that
\[
\| \P_x(X_{Z_{8t}}=\cdot) - \P_x(X_{U_{t} + Z_{8t}} = \cdot) \| \leq \|\P_x(Z_{8t}=\cdot) - \P_x(Z_{8t} + U_t = \cdot) \| \leq \frac{1}{8}.
\]
So, in the same way as in the proof of~\eqref{eq:upperbound} we obtain
\[
\|\P_x(X_{Z_{8t}}=\cdot) - \pi \| \leq \frac{3}{8}.
\]
Hence, we deduce that $\tg(3/8) \leq 8t$ and from Corollary~\ref{cor:submult} again there exists a positive constant $c'$ such that $\tg \leq c'\tg(3/8)$ and this finishes the proof.
\end{proof}

\section{A new proof of $\tprod \asymp \tl$ for reversible chains}\label{sec:newproof}

Recall the definition $\displaystyle \tprod= \max_{x,A} \pi(A)\E_x[\tau_A]$ from Remark~\ref{rem:aldous}. As noted there, Aldous \cite{Aldousmixing} showed the equivalence between the mixing time $\tct$ of a continuous time reversible chain and $\tprod$. Using the equivalence $\tl \asymp \tct$ (see \cite[Theorem~20.3]{LevPerWil}) it follows that for a reversible chain $\tprod \asymp \tl$. In this section we give a direct proof. Recall that $\tprod \geq c \tl$ for a reversible chain, where $c$ is a positive constant, follows from Theorem~\ref{thm:mix-hit}.

We will first state and prove a preliminary lemma, which is a variant of Kac's lemma (see for instance \cite[Lemma~21.13]{LevPerWil}). To that end we define for all $k$ and all sets $A$
\[
\tau^+_A = \min\{t \geq 1: X_t \in A \} \ \ \text{ and } \ \ \tau^{(k)}_A = \min\{t\geq k: X_t \in A \}.
\]

\begin{lemma}\label{lem:kac}
We have that
\[
\sum_{x \in A} \pi(x) \E_x[\tau^{(k)}_A] \leq k.
\]
\end{lemma}

\begin{proof}[\textbf{Proof}]
Let $\hat{P}$ be the transition matrix of the reversed chain, i.e.
\[
\hat{P}(x,y) = \frac{\pi(y)P(y,x)}{\pi(x)}.
\]
Then for all $t\geq k$ and $x_0,\ldots,x_t$ in the state space $S$, we have
\[
\pi(x_0) \prod_{i=1}^{t}P(x_{i-1},x_i) = \pi(x_t) \prod_{i=1}^{t}\hat{P}(x_i,x_{i-1}).
\]
Summing over all $x_0=x\in A,x_1 \in S,\ldots,x_{k-1} \in S,x_k \notin A,\ldots,x_{t-1}\notin A,x_t=y\in S$ we obtain
\begin{align}\label{eq:ineqkac}
\sum_{x \in A} \pi(x)\P_{x}(\tau^{(k)}_A \geq t) \leq \sum_{y}\pi(y) \hat{\P}_{y}(\hat{\tau}^+_A \in (t-k,t]),
\end{align}
where $\hat{\tau}_A^+$ stands for the first positive entrance time to $A$ for the reversed chain. Summing \eqref{eq:ineqkac} over all $t$ we get that
\begin{align*}
\sum_{x \in A}\pi(x) \E_x[\tau^{(k)}_A] &= \sum_t \sum_{x \in A} \pi(x) \P_{x}(\tau^{(k)}_A \geq t) \leq \sum_t\sum_y \pi(y) \sum_{s=t-k+1}^{t} \hat{\P}_y(\hat{\tau}^+_A=s)\\
&= \sum_y \pi(y)\sum_s \sum_{t=s}^{s+k-1}\hat{\P}_y(\hat{\tau}^+_A=s) = \sum_y \pi(y) \sum_s k \hat{\P}_y(\hat{\tau}_A^+=s)=k.
\end{align*}
\end{proof}

\begin{proof}[\textbf{Proof of $\tprod \leq c' \tl$}]
To simplify notation, let the chain $X$ be lazy and reversible.
From Lemma~\ref{lem:kac} and Markov's inequality it follows that for all $k$ and all sets $A$
\begin{align}\label{eq:statst}
\P_{\pi|_A}\left(\tau_A^{(k)} \geq \frac{2k}{\pi(A)}\right) \leq \frac{1}{2},
\end{align}
where $\pi|_A$ stands for the restriction of the stationary measure $\pi$ on $A$.

Take now $k=2\tl$. Then using submultiplicativity we get that $\displaystyle \bar{d}_L(k)\leq \bar{d}_L(\tl)^2\leq \tfrac{1}{4}$. Let $X_0 \sim \pi|_A$ and $z \in S$. Then
\begin{align*}
\|P_L^k(X_0,\cdot) - P_L^k(z,\cdot) \| \leq \frac{1}{4}.
\end{align*}
We can couple the two chains, $X_k,X_{k+1},\ldots$ with $X_0 \sim \pi|_A$ and $Y_k,Y_{k+1},\ldots$ with $Y_0=z$, so that they disagree with probability $\|P_L^k(X_0,\cdot) - P_L^k(z,\cdot) \|$.

Thus we obtain
\begin{align*}
\left| \P_{\pi|_A}\left(\tau_A^{(k)} \geq \frac{2k}{\pi(A)}\right) - \P_z\left(\tau_A^{(k)} \geq \frac{2k}{\pi(A)}\right)\right| & \leq \P_{\pi|_A,z}\left(\left\{\tau_A^{(k)}(X)\geq \frac{2k}{\pi(A)}\right\}
\bigtriangleup \left\{\tau_A^{(k)}(Y)\geq \frac{2k}{\pi(A)}\right\} \right)
\\ & \leq \P(\text{coupling fails}) \leq \frac{1}{4},
\end{align*}
and hence using \eqref{eq:statst} we get that
\[
\P_z\left(\tau_A^{(k)} \geq \frac{2k}{\pi(A)}\right) \leq \frac{3}{4}.
\]
Therefore for all $z$ we have that
\[
\P_z \left(\tau_A \geq \frac{2k}{\pi(A)}\right) \leq \P_z\left(\tau_A^{(k)} \geq \frac{2k}{\pi(A)}\right) \leq \frac{3}{4}.
\]
By performing independent experiments we see that $\tau_A$ is stochastically dominated by $\displaystyle \frac{2k}{\pi(A)}\mathrm{Geo}\left(\tfrac{3}{4}\right)$, where $\mathrm{Geo}$ stands for a Geometric random variable, and hence for all $z$ we get that
\[
\E_z\left[\tau_A\right] \leq \frac{8k}{3\pi(A)} = \frac{16\tl}{3\pi(A)}
\]
and this finishes the proof.
\end{proof}

\section{Application to robustness of mixing}\label{sec:robustness}

\begin{theorem}\label{thm:robust}
Let $T$ be a finite tree on $n$ vertices with unit conductances on the edges. Let $\widetilde{T}$ be a tree on the same set of vertices and edges as $T$ but with conductances on the edges satisfying
$c \leq c(x,y)\leq c'$, for all edges $e=\langle x,y \rangle$, where $c$ and $c'$ are two positive constants.
Then the mixing time of the lazy random walk on $T$ and on $\widetilde{T}$ are equivalent, i.e. in our notation, $\tl(T) \asymp \tl(\widetilde{T})$.
\end{theorem}

Before proving the theorem, we state and prove two lemmas which will be used in the proof but are also of independent interest.

\begin{lemma}\label{lem:hittingtree}
Let $T$ be a finite tree with edge conductances. For each subset $A$ of vertices and any vertex $v$ we have
\[
\max_{x} \E_x[\tau_A] \leq t_v \left(1 + \frac{1}{\pi(A)}\right),
\]
where $\tau_A$ stands for the first hitting time of $A$ by a simple random walk on $T$ and $t_v = \max_x \E_x[\tau_v]$.
\end{lemma}

\begin{proof}[\textbf{Proof}]
If $v \in A$, then the result is clear, so we assume that $v \notin A$. 

For all $x$ we have
\[
\E_x[\tau_A] \leq \E_x[\tau_v] + \E_v[\tau_A] \leq t_v + \E_v[\tau_A].
\]
Thus it suffices to show that
\begin{align}\label{eq:firstgoal}
\E_v[\tau_A] \leq \frac{t_v}{\pi(A)} .
\end{align}
In order to show that, we are going to look at excursions of the random walk from $v$.  
Defining $Z_A$ to be the time that the walk spends in $A$ in an excursion from $v$, i.e., $Z_A = \sum_{t=1}^{\tau_v^+} \1(X_t \in A)$, we can write
\begin{align*}
\P_v(\tau_A < \tau_v^+) = \frac{\E_v[Z_A]}{\E_v[Z_A| Z_A >0]}.
\end{align*}
Clearly
\begin{align*}
\E_v[Z_A] = \frac{\pi(A)}{\pi(v)} \ \ \text{and} \ \ \E_v[Z_A| Z_A >0] \leq t_v.
\end{align*}
Hence 
\begin{align*}
\P_v(\tau_A < \tau_v^+) \geq \frac{\pi(A)}{\pi(v)}\frac{1}{t_v}.
\end{align*}
Therefore we get 
\begin{align*}
\E_v[\tau_A] \leq \E_v\left[\sum_{i=1}^{N}\ell_i\right],
\end{align*}
where $N$ is a geometric random variable of success probability $\frac{\pi(A)}{\pi(v)}\frac{1}{t_v}$ and $\ell_i$ is the length of the $i$-th excursion from $v$. By Wald's identity we have
\begin{align*} 
\E_v[\tau_A] \leq \E_v[N] \E_v[\tau_v^+] \leq \frac{\pi(v)t_v}{\pi(A)} \frac{1}{\pi(v)} = \frac{t_v}{\pi(A)}
\end{align*}
and this completes the proof.
\end{proof}

We call a node $v$ in $T$ {\it{central}} if each component of $T-\{v\}$ has stationary probability at most $1/2$. It is easy to see that central nodes exist.
Indeed, for any node $u$ of the tree denote by $C(u)$ the component of $T-\{u\}$ with the largest stationary probability. Now consider the vertex $u^*$ that achieves $\displaystyle \min_u |\pi(C(u))|$. This is clearly a central node, since if $\pi(C(u^*)) >1/2$, then the neighbour $w \in C(u^*)$ of $u^*$ would satisfy $\pi(C(w)) < \pi(C(u^*))$, contradicting the choice of $u^*$.

\begin{lemma}\label{lem:lazytree}
Let $T$ be a tree on $n$ vertices with conductances on the edges. Then for any central node $v$ of $T$
\[
\tl\asymp t_v,
\]
where $t_v = \max_x \E_x[\tau_v]$.
\end{lemma}

\begin{proof}[\textbf{Proof}]
First of all from Lemma~\ref{lem:hittingtree} and Theorem~\ref{thm:mix-hit} we obtain that for any central node $v$
\begin{align}\label{eq:tltv}
\tl \leq c t_v,
\end{align}
for an absolute constant $c$. 

To finish the proof of the lemma we have to show that for any central node $v$ 
\begin{align}\label{eq:goalcentral}
\tl \geq c t_v,
\end{align}
for a positive absolute constant $c$. 
\newline
It is easy to see that $\E_x[\tau_v] = \E_x[\tau_B]$, for $x\neq v$, where $B$ is the union of $\{v\}$ and the components of $T-\{v\}$ that do not contain $x$.The definition of a central node gives that $\pi(B) \geq 1/2$. Hence, 
\begin{align}\label{eq:tvth12}
t_v \leq \thit(1/2).
\end{align}
Inequality~\eqref{eq:goalcentral} now follows from Theorem~\ref{thm:mix-hit}.
\end{proof}

We now recall a formula from~\cite[Lemma~1, Chapter~5]{AldFill} for the expected hitting time on trees.
\begin{lemma} \label{tree}
Let $T$ be a finite tree with edge conductances $c(u,v)$, for all edges $\langle u, v\rangle$. Let $x$ and $y$ be two vertices of $T$ and let $\{v_0=x,v_1,\ldots, v_{n}=y \}$ be the
unique path joining them.
Let $T_x(z)$ be the union of $\{z\}$ and the connected component of $T-\{z\}$ containing $x$.
Writing $C_i = \sum_{w,z \in T_x(v_{i+1})} c(w,z)$, we then have 
\[
E_x[\tau_y] = \sum_{i=0}^{n-1}\left( \frac{C_i}{c(v_i,v_{i+1})} - 1\right).
\]
\end{lemma}

\begin{proof}[\textbf{Proof of Theorem~\ref{thm:robust}}]
From Lemma~\ref{tree} and the boundedness of the conductances we get that for any two vertices $x$ and $v$ 
\[
\E_x[\tau_v] \asymp \E_x[\widetilde{\tau}_v],
\]
where $\widetilde{\tau}$ denotes hitting times for the random walk on $\widetilde{T}$.

Lemma~\ref{lem:lazytree} then finishes the proof.
\end{proof}

We end this section with another application of our results on the robustness of mixing when the probability of staying in place changes in a bounded way. The following corollary answers a question suggested to us by K.\ Burdzy (private communication).

\begin{corollary}
Let $P$ be an irreducible transition matrix on the state space $E$ and suppose that $(a(x,x))_{x\in E}$ satisfy $c_1\leq a(x,x) \leq c_2$ for all $x\in E$, where $c_1, c_2 \in (0,1)$. Let $Q$ be the transition matrix of the Markov chain with transitions: when at $x$ it stays at $x$ with probability $a(x,x)$. Otherwise, with probability $1-a(x,x)$ it jumps to a state $y\in E$ with probability $P(x,y)$. We then have
\[
\tmix(Q) \asymp \tl,
\]
where $\tmix(Q)$ is the mixing time of the transition matrix $Q$.
\end{corollary}

\begin{proof}[\textbf {Proof}]

Since the loop probabilities $a(x,x)$ are bounded from below and above, it follows that if $\til{\pi}$ is the stationary probability of the matrix $Q$, then $\til{\pi} \asymp \pi$. As we noted in the Introduction, the lower bound of Theorem~\ref{thm:mix-hit} is always true and thus we have
\begin{align}\label{eq:firsttil}
\max_{x,A: \til{\pi}(A) \geq 1/4 } \estart{\til{\tau}_A}{x} \lesssim  \tmix(Q).
\end{align}
For every $y\in E$ let $(\xi_i^{(y)})_{i\in \N}$ be i.i.d.\ geometric random variables of mean $1/a(y,y)$. Then we can write
\begin{align}\label{eq:laziness}
\til{\tau}_A = \sum_{y\in E} \sum_{i=1}^{L_y} \xi_i^{(y)},
\end{align}
where $L_y$ is the local time at $y$ up to the first hitting time of $A$ by the chain with transition matrix~$P$. Wald's identity gives
\[
\estart{\til{\tau}_A}{x} = \sum_{y\in E} \frac{\estart{L_y}{x}}{a(y,y)}. 
\]
If $\tau'_A$ is the hitting of $A$ by the lazy version of the chain, i.e.\ taking $a(y,y) = 1/2$ for all $y$, then using the assumption on the boundedness of the probabilities $(a(y,y))$ we get
\[
\estart{\til{\tau}_A}{x} \asymp \estart{\tau'_A}{x}.
\]
From~\eqref{eq:laziness} applying Wald's identity again we deduce
\[
\estart{\tau'_A}{x} = 2\estart{\tau_A}{x},
\]
where $\tau_A$ is the first hitting time of the set $A$ by the Markov chain with transition matrix $P$.
Hence using Theorem~\ref{thm:mix-hit} and~\eqref{eq:firsttil} we deduce that 
\[
\tmix(Q) \gtrsim \tl.
\]
It remains to show
\begin{align}\label{eq:finalgoaltmix}
\tmix(Q) \lesssim \tl.
\end{align}
Using Proposition~\ref{prop:tavertmix} we get that 
\[
\tmix(A) \lesssim \taver.
\]
This together with Theorem~\ref{thm:ave-lazy} finishes the proof of~\eqref{eq:finalgoaltmix}.
\end{proof}

\section{Examples and Questions}

We start this section with examples that show that the reversibility assumption in Theorem~\ref{thm:mix-hit} and Corollary~\ref{cor:tltstop} is essential.

\begin{example}\rm{{Biased random walk on the cycle.}
\newline
Let $\Z_n=\{1,2,\ldots,n\}$ denote the $n$-cycle and let $\displaystyle P(i,i+1)=\tfrac{2}{3}$ for all $1 \leq i <n $ and $\displaystyle P(n,1)=\tfrac{2}{3}$. Also $\displaystyle P(i,i-1)=\tfrac{1}{3}$, for all $1<i\leq n$, and $\displaystyle P(1,n)=\tfrac{1}{3}$. Then it is easy to see that the mixing time of the lazy random walk is of order $n^2$, while the maximum hitting time of large sets is of order $n$. Also, in this case $\tstop = O(n)$, since for any starting point, the stopping time that chooses a random target according to the stationary distribution and waits until it hits it, is stationary and has mean of order $n$. This example demonstrates that for non-reversible chains, $\thit$ and $\tstop$ can be much smaller than $\tl$.
}
\end{example}

\begin{example}\rm{{The greasy ladder.}
\newline
Let $S=\{1,\ldots,n\}$ and $\displaystyle P(i,i+1)=\tfrac{1}{2}=1-P(i,1)$ for $i=1,\ldots,n-1$ and $P(n,1)=1$. Then it is easy to check that
\[
\pi(i)=\frac{2^{-i}}{1-2^{-n}}
\]
is the stationary distribution and that $\tl$ and $\thit$ are both of order 1.
\newline
This example was presented in Aldous \cite{Aldousmixing}, who wrote that $\tstop$ is of order $n$. We give an easy proof here. Essentially the same example is discussed by Lov{\'a}sz and Winkler \cite{LWmixing} under the name ``the winning streak''.
\newline
Let $\tau_\pi$ be the first hitting time of a stationary target, i.e.\ a target chosen according to the stationary distribution. Then starting from $1$, this stopping time achieves the minimum in the definition of $\tstop$, i.e.\
\[
\E_1[\tau_\pi] = \min\{\E_1[\Lambda]: \Lambda \text{ is a stopping time s.t. } \P_1(X_{\Lambda}\in \cdot)=\pi(\cdot) \}.
\]
Indeed, starting from $1$ the stopping time $\tau_\pi$ has a halting state, which is $n$, and hence from Theorem~\ref{thm:lovwin} we get the mean optimality. By the random target lemma \cite{AldFill} and \cite{LevPerWil}
we get that $\E_i[\tau_\pi]=\E_1[\tau_\pi]$, for all $i \leq n$. Since for all $i$ we have that
\[
\E_i[\tau_\pi] \geq \min\{\E_i[\Lambda]: \Lambda \text{ is a stopping time s.t. } \P_i(X_{\Lambda}\in \cdot)=\pi(\cdot) \},
\]
it follows that $\tstop \leq \E_1[\tau_\pi]$. But also $\E_1[\tau_\pi] \leq \tstop$, and hence $\tstop = \E_1[\tau_\pi]$. By straightforward calculations, we get that $\E_1[T_i] = 2^{i}(1-2^{-n})$, for all $i\geq 2$, and hence
\[
\tstop = \E_1[\tau_\pi] = \sum_{i=2}^{n} 2^{i}(1- 2^{-n})\frac{2^{-i}}{1-2^{-n}} = n - 1.
\]
This example shows that for a non-reversible chain $\tstop$ can be much bigger than
$\tl$ or $\thit$.
}
\end{example}

\begin{question}\rm{
The equivalence $\thit(\alpha) \asymp \tl$ in Theorem~\ref{thm:mix-hit} is not valid for $\displaystyle \alpha > \tfrac{1}{2}$, since for two $n$-vertex complete graphs with a single edge connecting them, $\tl$ is of order $n^2$ and $\thit(\alpha)$ is at most $n$ for any $\alpha>1/2$.
Does the equivalence $\displaystyle \thit\left(1/2\right) \asymp \tl$ hold for all reversible chains?}
\newline 
 (After this question was posed in the first version of this paper, it was answered positively by Griffiths et al~\cite{SimonRoberto}.)
\end{question}

\section*{Acknowledgments}
We are indebted to Oded Schramm for suggesting the use of the parameter $\tg$ to relate mixing times and hitting times. We are grateful to David Aldous for helpful discussions.
After this work was completed, we were informed that Theorem~\ref{thm:mix-hit} was also proved independently by Roberto Imbuzeiro Oliveira (\cite{Imbuzeiro}).

We thank Yang Cai, J{\'u}lia Komj{\'a}thy and the referee for useful comments.

\bibliographystyle{plain}
\bibliography{biblio}

\end{document}